\documentclass{birkjour}
\usepackage{amsfonts,amssymb,amsmath,amsgen,amsopn,amsbsy}
 \usepackage{amssymb,stmaryrd}
\usepackage{amsmath}
\usepackage{graphicx}
\usepackage{epstopdf}
\usepackage{xcolor}

 \numberwithin{equation}{section}
\newtheorem{theorem}{Theorem}[section]

\newtheorem{lemma}[theorem]{Lemma}
\newtheorem{proposition}[theorem]{Proposition}

\theoremstyle{definition}
\newtheorem{definition}[theorem]{Definition}

\theoremstyle{remark}

\begin{document}

%
%
%
%
%
%
%
%
%

\setlength{\headheight}{26pt}
\title[Positive Toeplitz Operators]
 {Positive Toeplitz Operators from \\ a Weighted Harmonic Bloch Space \\ into Another}

\author[\"{O}mer Faruk Do\u{g}an]{\"{O}mer Faruk Do\u{g}an }



\date{January 1, 2004}
\address{Department of Mathematics, Tek$\dot{\hbox{\i}}$rda\u{g} \\ Namik Kemal University,
59030 \\ Tek$\dot{\hbox{\i}}$rda\u{g}, Turkey,}

\email{ofdogan@nku.edu.tr}

\begin{abstract}
We define positive Toeplitz operators between weighted harmonic Bloch spaces $b^\infty_\alpha$ on the unit ball of $\mathbb{R}^n$ for the full range of parameter $\alpha\in\mathbb{R}$. We give characterizations of bounded and compact Toeplitz operators taking one weighted harmonic  Bloch space  into a\-not\-her in terms of Carleson and vanishing Carleson measures.
\end{abstract}

\maketitle

\section{Introduction}\label{s-introduction}

Let $n\geq 2$ be an integer and $\mathbb{B}=\mathbb{B}_{n}$ be the open unit ball in $\mathbb{R}^n$.
Let $\nu$  be the Lebesgue volume  measure on $\mathbb{B}$ normalized so that $\nu(\mathbb{B})=1$.
For  $\alpha\in \mathbb{R}$, we define the weighted volume measures $\nu_\alpha$ on $\mathbb{B}$ by
\[
d\nu_\alpha(x)=\frac{1}{V_\alpha} (1-|x|^2)^\alpha d\nu(x).
\]
These measures are finite when $\alpha>-1$ and in this case we choose $V_\alpha$ so that $\nu_\alpha(\mathbb{B})=1$. Naturally $V_0=1$.  For $\alpha\leq -1$, we set $V_\alpha=1$. We denote the Lebesgue classes with respect to $\nu_\alpha$ by $L^p_{\alpha}$,  $0<p<\infty$, and the corresponding norms by $\|\cdot\|_{L^p_{\alpha}}$.

Let $h(\mathbb{B})$ be the space of all complex-valued harmonic functions on $\mathbb{B}$ with the topology of uniform convergence on compact subsets.  For $0<p<\infty$ and $\alpha>-1$, the harmonic weighted Bergman space $b^p_\alpha$ is defined by $b^p_\alpha=  L^p_\alpha \cap h(\mathbb{B})$ endowed with the norm
$\|\cdot\|_{L^p_{\alpha}}$.  The subfamily $b^2_\alpha$ is a  reproducing kernel Hilbert space with respect to the inner product $[f,g]_{b^2_\alpha}=\int_{\mathbb{B}}f\overline{g} \, d\nu_{\alpha}(x)$ and with the reproducing kernel $R_\alpha(x,y)$  such that $f(x)=[f,R_\alpha(x,\cdot)]_{b^2_\alpha}$ for every $f\in b^2_\alpha$ and $x\in \mathbb{B}$. It is well-known that $R_\alpha$ is real-valued  and $R_\alpha(x,y)=R_\alpha(y,x)$. The homogeneous expansion of $R_\alpha(x,y)$ is given in the $\alpha>-1$ part of the formulas (\ref{Rq - Series expansion}) and (\ref{gamma k q-Definition}) below (see \cite{DS}, \cite{GKU2}).

The well-known harmonic Bloch space $b$ is the space of all $f\in h(\mathbb{B})$ such that
\begin{equation*}
\sup_{x\in \mathbb{B}} (1-|x|^2) |\nabla f(x)| < \infty.
\end{equation*}
We denote by $L^{\infty}$ the Lebesgue class of essentially bounded functions on $\mathbb{B}$, and for $\alpha \in \mathbb{R}$ we define
\begin{equation*}
L^{\infty}_\alpha = \{\varphi : (1-|x|^2)^{\alpha} \varphi(x) \in L^{\infty} \},
\end{equation*}
so that $L^{\infty}_0 = L^{\infty}$. The norm on $L^{\infty}_\alpha$ is
\begin{equation*}
\|\varphi\|_{L^{\infty}_\alpha} = \|(1-|x|^2)^\alpha \varphi(x)\|_{L^\infty}.
\end{equation*}
 For $\alpha>0$, the weighted harmonic Bloch space $b_\alpha$ is $h(\mathbb{B})\cap L^\infty_\alpha$ endowed with the norm $ \|\cdot\|_{L^{\infty}_\alpha}$.

For $\alpha>0$, the weighted harmonic Bloch space $b^\infty_\alpha$  is naturally imbedded in  $  L^\infty_\alpha$ by the inclusion $i$.
For $\alpha>0$, the harmonic Bergman projection  $ Q_\alpha:L^\infty_\alpha \to b^\infty_\alpha$ is given by the integral operator

\begin{equation}\label{orthogonal projection}
  Q_\alpha f(x)= \frac{1}{V_\alpha} \int_{\mathbb{B}} R_\alpha(x,y) f(y) (1-|y|^2)^\alpha d\nu(y) \quad (f\in L^\infty_\alpha ).
\end{equation}
This integral operator plays a major role in the theory of weighted harmonic
Bloch spaces and the question when the projection $ Q_\alpha:L^\infty_\beta \to b^\infty_\beta$  is bounded is studied in many sources such as \cite{CKY1,JP,L} for $\beta =0$ and \cite{RK} with a different integral
representation valid for $\beta >-1$.  Then one defines the Toeplitz operator  ${_{\alpha}}T_{\phi} :b^\infty_\alpha \to b^\infty_\alpha$ with symbol $\phi\in L^{1}$ by ${_{\alpha}}T_{\phi}=Q_{\alpha}M_{\phi}i$, where $M_{\phi}$ is the operator of multiplication by $\phi$. Let $\mu$ be a finite complex Borel measure on $\mathbb{B}$. The Toeplitz operator ${_{\alpha}}T_{\mu}$ with symbol $\mu$ is defined by
\[
{_{\alpha}}T_{\mu}f(x) = \frac{1}{V_\alpha}\int_{\mathbb{B}} R_\alpha(x,y)  f(y) (1-|y|^2)^\alpha  d\mu(y)
\]
for $f \in L^{\infty}_{\alpha}$. The operator ${_{\alpha}}T_{\mu}$ is more general and reduces to ${_{\alpha}}T_{\phi}$  when $d\mu=\phi d\nu$. Toeplitz operators have been studied extensively on the harmonic Bergman spaces by many
authors. Especially, positive symbols of bounded and compact Toeplitz operators are completely characterized in term of Carleson measures as in \cite{M},\cite{M2} on the ball and in \cite{CLN1} on smoothly bounded domains. Toeplitz operators from a harmonic Bergman space into another are considered and positive symbols of bounded and compact Toeplitz operators are characterized in \cite{CLN2} on smoothly bounded domains and in \cite{CKY} on the half space.

The weighted harmonic Bergman $b^p_\alpha$  and Bloch $b^\infty_\alpha$ spaces
can be extended to the whole range $\alpha \in \mathbb{R}$. These are studied in detail in \cite{GKU2} and \cite{DU1}, respectively. We call the extended family $b^p_\alpha$ $(\alpha \in \mathbb{R})$  harmonic Bergman-Besov spaces and the corresponding reproducing kernels $R_\alpha(x,y)$ $(\alpha \in \mathbb{R})$  harmonic Bergman-Besov kernels. The homogeneous expansion of  $R_\alpha(x,y)$ can be expressed in terms of zonal harmonics
\begin{equation}\label{Rq - Series expansion}
R_\alpha(x,y)=\sum_{k=0}^{\infty} \gamma_k(\alpha) Z_k(x,y) \quad  (\alpha\in \mathbb{R}, \, x,y\in \mathbb{B}),
\end{equation}
where (see \cite[Theorem 3.7]{GKU1}, \cite[Theorem 1.3]{GKU2})

\begin{equation}\label{gamma k q-Definition}
        \gamma_k(\alpha):= \begin{cases}
         \dfrac{(1+n/2+\alpha)_k}{(n/2)_k}, &\text{if $\, \alpha > -(1+n/2)$}; \\
         \noalign{\medskip}
         \dfrac{(k!)^2}{(1-(n/2+\alpha))_k (n/2)_k}, &\text{if $\, \alpha \leq -(1+n/2)$},
\end{cases}
\end{equation}
and $(a)_b$ is the Pochhammer symbol. For definition and details about $Z_k(x,y)$, see \cite[Chapter 5]{ABR}.

The spaces $b^\infty_\alpha$  can  be defined by using the radial differential operators $D^t_s$ $(s,t \in \mathbb{R})$ introduced in \cite{GKU1} and \cite{GKU2}. These operators are defined in terms of reproducing kernels of harmonic Besov spaces and still mapping  $h(\mathbb{B})$ onto itself.  The properties of $D^t_s$ will be reviewed in Section \ref{s-preliminaries}. Consider the linear transformation $I_{s}^{t}$ defined for $f\in h(\mathbb{B})$ by

\begin{equation*}
  I^t_s f(x) := (1-|x|^2)^t D^t_s f(x).
\end{equation*}

\begin{definition}\label{definition of the h B-B space}
For  $\alpha \in \mathbb{R}$, we define the weighted harmonic Bloch space $b^\infty_\alpha$ to consist of all $f\in h(\mathbb{B})$ for which $ I^t_s f$
belongs to  $L^\infty_\alpha$ for some  $s,t$ satisfying (see \cite{DU1} )

\begin{equation}\label{alpha+pt}
 \alpha+t>0.
 \end{equation}
The quantity
\[
\|f\|_{b^\infty_\alpha} = \| I^t_s f\|_{L^\infty_\alpha}=\sup_{x\in \mathbb{B}}\, (1-|x|^2)^{\alpha+t} |D^t_s f(x)|<\infty,
\]
defines a norm on $b^\infty_\alpha$ for any such $s,t\in \mathbb{R}$.
\end{definition}
It is well-known that the above definition is independent of $s,t$ under (\ref{alpha+pt}),  and the norms  on a given space are all equivalent. Thus for a given pair $s,t$,  $ I^t_s$ isometrically imbeds $b^\infty_\alpha$ into $L^\infty_\alpha$ if and only if (\ref{alpha+pt}) holds.

Harmonic Bergman-Besov projections $Q_{s}$ that map Lebesgue classes boundedly onto weighted Bloch spaces $b^\infty_\alpha$ can be precisely identified as  in the case of $\alpha>0$ by
\begin{equation}\label{two}
s>\alpha-1.
\end{equation}
Then $I_{s}^{t}$ is a right inverse to $Q_{s}$. This is all done in \cite{DU1}.

Now let $\alpha\in \mathbb{R}$,  s and t satisfing (\ref{two}) and (\ref{alpha+pt}), and a measurable function $\phi$ on $\mathbb{B}$ be given. Harmonic Bergman-Besov projections $Q_{s}$ forces us to define Toeplitz operators on all $b^\infty_\alpha $ as follows.  We define the Toeplitz operator ${_{s,t}}T_{\phi}: b^\infty_\alpha \to b^\infty_\alpha $ with symbol $\phi$ by ${_{s,t}}T_{\phi}=Q_{s}M_{\phi}I_{s}^{t}$. Explicitly,
\[
{_{s,t}}T_{\phi}f(x) = \int_{\mathbb{B}} R_s(x,y) \phi(y) I^t_s f(y)  d\nu_{s}(y) \quad (f\in b^\infty_\alpha).
\]
We see that ${_{s,t}}T_{\phi}$ makes sense if $\phi \in L^{1}_{s-\alpha}$.
When $\alpha>0$, one can choose $t=0$ and a value of $s$ satisfying (\ref{two})  is $s=\alpha$. Then $I_{\alpha}^{0}$ is inclusion, and  ${_{s,t}}T_{\phi}$  reduces to the classical Toeplitz operator ${_{\alpha}}T_{\phi}=Q_{\alpha}M_{\phi}i$ on the harmonic weighted Bloch spaces  $b^\infty_\alpha $, $\alpha>0$. We use the term  classical to mean a Toeplitz operator with $i=I_{\alpha}^{0}$.   It is possible to take $s \neq \alpha$ also when $\alpha>-1$. So we have more general Toeplitz operators defined via $I_{s}^{t}$  strictly on harmonic Bloch spaces too. It turns out that the properties of Toeplitz operators studied in this paper are independent of $s,t$ under (\ref{two}) and (\ref{alpha+pt}).

Having obtained the integral form for ${_{s,t}}T_{\phi}$, we can now define Toeplitz operators on $b^\infty_\alpha $ with symbol $\mu$. Let $\alpha$, and $s$ and $t$ satisfing (\ref{two}) and (\ref{alpha+pt}) be given. We let
\[
d\kappa(y)=(1-|y|^{2})^{s+t}d\mu(y)
\]
and define
\begin{align*}
{_{s,t}}T_{\mu}f(x) &= \frac{1}{V_{s}} \int_{\mathbb{B}} R_s(x,y) I^t_s f(y) (1-|y|^{2})^{s}d\mu(y)\\
&=\frac{1}{V_{s}} \int_{\mathbb{B}} R_s(x,y) D^t_s f(y) d\kappa(y).
\end{align*}
The operator ${_{s,t}}T_{\mu}$ is more general and reduces to ${_{s,t}}T_{\phi}$ when $d\mu=\phi d\nu$.   It makes sense when
\[
d\psi(y)=(1-|y|^{2})^{-(\alpha+t)}d\kappa(y)=(1-|y|^{2})^{s-\alpha}d\mu(y)
\]
is finite. Note that $\mu$ need not be finite in conformity with that $\alpha$ is unrestricted.

 The \textit{holomorphic} analogues of our characterizations for Dirichlet spaces have been obtained   in \cite{AK}; these results are all concerned with Toeplitz o\-pe\-ra\-tors from a Dirichlet space into itself.
Recently, in \cite{DOG2}, positive symbols of bounded and compact general Toeplitz operators between  harmonic Bergman-Besov spaces are  completely characterized in term of Carleson measures. In this paper, we consider the Toeplitz operator ${_{s,t}}T_{\mu}$ with positive symbol and characterize those that are bounded and compact from  a weighted harmonic Bloch space $b^{\infty}_{\alpha_{1}}$ into another $b^{\infty}_{\alpha_{2}}$ for $\alpha_{1},\alpha_{2} \in \mathbb{R}$. Our main tool is Carleson measure.

 Let $\mu$ be a positive Borel measure on $\mathbb{B}$. For $\alpha>-1$, we say that $\mu$ is a   $\alpha$-Carleson measure if  the inclusion $i: b^p_\alpha\to L^p(\mu)$ is bounded, that is, if
\[
\left(\int_{\mathbb{B}} |f(x)|^p\, d\mu(x) \right)^{1/p} \lesssim \|f\|_{b^p_\alpha}, \qquad (f\in b^p_\alpha).
\]
We also denote
\begin{equation*}
 \| \mu\|_{\alpha}=\sup_{f\in b^p_\alpha, \|f\|_{b^p_\alpha}\leq1} \int_{\mathbb{B}} |f(x)|^p\, d\mu(x).
\end{equation*}
As is common with Carleson measure theorems, the property of being an $\alpha$-Carleson measure is
independent of $p$, because  Theorem \ref{Carleson-Besov2} is true for any $p$.
 However, it depends on $\alpha>-1$. So for a fixed $\alpha$, an $\alpha$-Carleson
measure for one $b^p_\alpha$ is a Carleson measure for all $b^p_\alpha$  with the same $\alpha$.
We can now state our main result.

\begin{theorem}\label{Theorem-1}
Let  $\alpha_{1},\alpha_{2} \in \mathbb{R}$. Suppose that $\alpha_{1}+t>0$, $\alpha_{2}+t>0$ and

\begin{equation} \label{Equ-1}
s>\alpha_{i}-1, \qquad i=1,2.
\end{equation}
Let
\[
 \gamma=s+t+\alpha_{1}-\alpha_{2}.
\]
Let $\mu$ be a positive Borel measure on $\mathbb{B}$ and $d\kappa(y)=
(1-|x|^{2})^{s+t}d\mu(y)$. Then the following statements are equivalent:
\begin{enumerate}
\item[(i)] ${_{s,t}}T_{\mu}$ is bounded  from  $b^{\infty}_{\alpha_{1}}$ to $b^{\infty}_{\alpha_{2}}$.
\item[(ii)] $\kappa$ is a $\gamma$-Carleson measure.
\end{enumerate}
\end{theorem}

 In order to characterize compact positive Toeplitz operators ${_{s,t}}T_{\mu}$ from  weighted harmonic Bloch spaces $b^{\infty}_{\alpha_{1}}$ into another $b^{\infty}_{\alpha_{2}}$ for all  $\alpha_{1},\alpha_{2} \in \mathbb{R}$, we introduce the notion of vanishing $\alpha$-Carleson measures. We say that $\mu\geq 0$ is a vanishing $\alpha$-Carleson measure if for any sequence $\{f_{k}\}$ in $b^{p}_{\alpha}$ with $f_{k}\to 0$ uniformly on each compact subset of $\mathbb{B}$ and $ \|f_{k}\|_{b^{p}_{\alpha}}\leq 1$,

\[
\lim_{k\to \infty}\int_{\mathbb{B}} |f_{k}(x)|^p\, d\mu(x) =0.
\]
One can see from Theorem \ref{Carleson-Besov3}  that the notion of vanishing $\alpha$-Carleson measures
on $b^{p}_{\alpha}$ is also independent of $p$.
\begin{theorem}\label{Theorem-2}
Let  $\alpha_{1},\alpha_{2} \in \mathbb{R}$. Let $s,t$, and $\gamma$ be as in Theorem \ref{Theorem-1}.
Let $\mu$ be a positive Borel measure on $\mathbb{B}$ and $d\kappa(y)=
(1-|x|^{2})^{s+t}d\mu(y)$. Then the following statements are equivalent:
\begin{enumerate}
\item[(i)] ${_{s,t}}T_{\mu}$ is compact  from  $b^{\infty}_{\alpha_{1}}$ to $b^{\infty}_{\alpha_{2}}$.
\item[(ii)] $\kappa$ is a vanishing $\gamma$-Carleson measure.
\end{enumerate}
\end{theorem}

The proofs of our results are inspired by the work of Pau and Zhao in \cite{PZ}, where bounded and compact  classical Toeplitz
operators between holomorphic weighted Bergman spaces are characterized.

The paper is organized as follows. The notation and some preliminary results
are summarized in Section  \ref{s-preliminaries}.   We will recall various characterizations of (vanishing) $\alpha$-Carleson measures for weighted harmonic Bergman spaces in Section \ref{s-carleson}.
Section \ref{proof1} is devoted to the proof of our main results, Theorem \ref{Theorem-1} and \ref{Theorem-2}.

In the following for two positive expressions $X$ and $Y$ we write $X\lesssim Y$ if there exists a positive constant $C$, whose exact value is inessential, such that $X\leq CY$. If both $X\lesssim Y$ and $Y\lesssim X$, we write $X\sim Y$.
\section{Preliminaries}\label{s-preliminaries}

In this section we collect some known facts that will be used  throughout the paper.

The Pochhammer symbol $(a)_b$ is defined by
\[
(a)_b=\frac{\Gamma(a+b)}{\Gamma(a)},
\]
when $a$ and $a+b$ are off the pole set $-\mathbb{N}$ of the gamma function. By Stirling formula,
\begin{equation}\label{Stirling}
\frac{(a)_c}{(b)_c} \sim c^{a-b} \quad (c\to\infty).
\end{equation}

A harmonic function $f$ on $\mathbb{B}$ has a homogeneous expansion, that is, there exist homogeneous harmonic polynomials $f_k$ of degree $k$ such that $f(x)=\sum_{k=0}^\infty f_k(x)$. The series uniformly and absolutely converges on compact subsets of $\mathbb{B}$.

\subsection{Pseudohyperbolic metric}

The canonical M\"obius transformation on $\mathbb{B}$ that exchanges $a$ and $0$ is
\[
\varphi_a(x) = \frac{(1-|a|^2) (a-x) + |a-x|^2 a}{[x,a]^2}.
\]
Here the bracket $[x,a]$ is defined by
\[
[x,a]=\sqrt{1-2 x\cdot a+|x|^2 |a|^2},
\]
where $x\cdot a$ denotes the inner product of $x$ and $a$ in $\mathbb{R}^n.$ The pseudohyperbolic distance between $x,y\in \mathbb{B}$ is
\[
\rho(x,y)=|\varphi_x(y)|=\frac{|x-y|}{[x,y]}.
\]

For a proof of the following lemma see \cite[Lemma 2.2]{CKL}.

\begin{lemma}\label{Inequality}
Let $a,x,y\in \mathbb{B}$. Then

\[
\frac{1-\rho(x,y)}{1+\rho(x,y)} \leq \frac{[x,a]}{[y,a]} \leq \frac{1+\rho(x,y)}{1-\rho(x,y)}.
\]
\end{lemma}

The following two lemmas show that if $x,y\in\mathbb{B}$ are close in the pseudohyperbolic metric, then certain quantities are comparable. Both of them easily follow from Lemma \ref{Inequality} (note that $[x,x]=1-|x|^2$).

\begin{lemma}\label{x-y-close}
Let $0<\delta<1$. Then
\[
[x,y]\sim 1-|x|^2 \sim 1-|y|^2,
\]
for all $x,y\in \mathbb{B}$ with $\rho(x,y)<\delta$.
\end{lemma}

\begin{lemma}\label{Bracket-Hyper}
Let $0<\delta<1$. Then
\[
[x,a] \sim [y,a],
\]
for all $a,x,y\in \mathbb{B}$ with $\rho(x,y)<\delta$.
\end{lemma}

For $0<\delta<1$ and $x\in\mathbb{B}$ we denote the pseudohyperbolic ball with center $x$ and radius $\delta$ by $E_\delta(x)$. The pseudohyperbolic ball $E_\delta(x)$ is also a Euclidean ball with center $c$ and radius $r$, where
\[
c=\frac{(1-\delta^2)x}{1-\delta^2|x|^2} \qquad \text{and} \qquad r=\frac{(1-|x|^2)\delta}{1-\delta^2|x|^2}.
\]
It follows that for fixed $0<\delta<1$, we have $\nu(E_\delta(x))\sim (1-|x|^2)^n$. More generally, for $\alpha\in \mathbb{R}$, by Lemma \ref{x-y-close}
\begin{equation}\label{hyper-volume}
\nu_\alpha(E_\delta(x))=\frac{1}{V_\alpha}  \int_{E_\delta(x)} (1-|y|^2)^\alpha \, d\nu(y) \sim (1-|x|^2)^\alpha \nu(E_\delta(x)) \sim (1-|x|^2)^{\alpha+n}.
\end{equation}

Let $\{a_k\}$ be a sequence of points in $\mathbb{B}$ and $0<\delta<1$. We say that $\{a_k\}$ is $\delta$-separated if $\rho(a_j,a_k)\geq\delta$ for all $j\neq k$. For a proof of the following lemma see, for example, \cite{L2}.
\begin{lemma}\label{ak}
  Let $0<\delta<1$. There exists a sequence of points $\{a_k\}$ in $\mathbb{B}$ satisfying the following properties:
  \begin{enumerate}
    \item[(i)] $\{a_k\}$ is $\delta$-separated.
    \item[(ii)] $\displaystyle \bigcup_{k=1}^\infty E_\delta(a_k) = \mathbb{B}$.
    \item[(iii)] There exists a positive integer $N$ such that every $x\in \mathbb{B}$ belongs to at most $N$ of the balls $E_\delta(a_k)$.
  \end{enumerate}
\end{lemma}

In what follows whenever we use expressions like $\widehat{\mu}_{\alpha,\delta}(a_{k})$, the sequence $\{a_k\}=\{a_k(\delta)\}$ will always refer to the sequence chosen in Lemma \ref{ak}.

We also need to the following pointwise estimate. See
\cite[Lemma 3.1]{DOG} for a proof.
\begin{lemma}\label{growth}
Let $0<p<\infty$ and $\alpha>-1$. Then
\[
|u(x)| \lesssim \frac{\|u\|_{b^p_\alpha}}{(1-|x|^2)^{(n+\alpha)/p}}
\]
for all $u\in b^p_\alpha$ and $x\in \mathbb{B}$.
\end{lemma}

\subsection{Reproducing Kernels and the Operators $D^t_s$}
For every $\alpha\in \mathbb{R}$ we have $\gamma_{0} (\alpha)=1$, and therefore
\begin{equation*}
R_\alpha(x,0)=R_\alpha(0,y)=1, \quad (x,y\in \mathbb{B}, \alpha\in \mathbb{R}).
\end{equation*}
Checking the two cases in (\ref{gamma k q-Definition}), we have by (\ref{Stirling})
\begin{equation}\label{gamma-k-asymptotic}
\gamma_k(\alpha) \sim k^{1+\alpha} \quad (k\to \infty).
\end{equation}
$R_\alpha(x,y)$ is harmonic as a function of either of its variables on $\overline{\mathbb{B}}$.
Using the coefficients in the extended kernels we define the radial differential operators $D^t_s$.

\begin{definition}
Let $f=\sum_{k=0}^\infty f_k\in h(\mathbb{B})$ be given by its homogeneous expansion. For $s,t\in\mathbb{R}$ we define  $D_s^t$ on $ h(\mathbb{B}) $  by

\begin{equation*}
  D_s^t f := \sum_{k=0}^\infty \frac{\gamma_k(s+t)}{\gamma_k(s)} \, f_k.
\end{equation*}

\end{definition}
By (\ref{gamma-k-asymptotic}), $\gamma_k(s+t)/\gamma_k(s) \sim k^t$ for any $s,t$ and, roughly speaking, $D_s^t$ multiplies the $k$th homogeneous part of $f$ by $k^{t}$. For every $s\in \mathbb{R}$, $D_s^0=I$, the identity. An important property of $D^t_s$ is that it is invertible with two-sided inverse $D_{s+t}^{-t}$:

\begin{equation}\label{*}
D^{-t}_{s+t} D^t_s = D^t_s D^{-t}_{s+t} = I,
\end{equation}
which follows from the additive property $D_{s+t}^{z} D_s^t = D_s^{z+t}$.

For every $s,t \in \mathbb{R}$, the map $D^t_s: h(\mathbb{B})\to h(\mathbb{B})$ is continuous in the topology
of uniform convergence on compact subsets (see \cite[Theorem 3.2]{GKU2}). The parameter $s$ plays a minor role and is used to have the precise relation

\begin{equation}\label{**}
D_s^t R_s(x,y)=R_{s+t}(x,y)
\end{equation}

One of the most important properties about the operator $D^t_s$ is that it allows us to pass from one Bloch space to another. More precisely, we have the following isomorphism. For a proof see \cite[Proposition 4.6]{DU1}.

\begin{lemma}\label{Apply-Dst}
Let  $\alpha,s,t\in \mathbb{R}$.
The map $D^t_s:b^\infty_\alpha \to b^\infty_{\alpha+t}$ is an isomorphism.
\end{lemma}

 The following duality result is \cite[Theorem 5.4]{DU1}.

\begin{theorem}\label{Theorem-Dual-of-b1q}
Let $q \in \mathbb{R}$. Pick $s',t'$ such that
\begin{align*}
s' &> q, \\
q+t' &> -1.
\end{align*}
The dual of $b^1_{q}$ can be identified with $b_\alpha$ (for any $\alpha \in \mathbb{R}$) under the pairing
\begin{equation*}
\langle f, g \rangle = \int_{\mathbb{B}} I^{t'}_{s'} f \, \overline{I^{s'-q-\alpha}_{t'+q+\alpha}g} \ d\nu_{q+\alpha}, \qquad (f\in b^1_{q}, \ g\in b_\alpha).
\end{equation*}
\end{theorem}

\subsection{Estimates of Reproducing Kernels}

In case $\alpha>-1$, the reproducing kernels $R_\alpha(x,y)$ are well-studied by various authors. We recall some of their properties below. For extension of these properties to $\alpha\in\mathbb{R}$ we refer to \cite{GKU2}.

For a proof of the following pointwise estimate see \cite{CR,R} when $\alpha>-1$ and \cite{GKU2} when $\alpha\in\mathbb{R}$.

\begin{lemma}\label{R-alpha}
Let $\alpha\in\mathbb{R}$. For all $x,y\in\mathbb{B}$,
\[
|R_\alpha(x,y)| \lesssim
\begin{cases}
\dfrac{1}{[x,y]^{\alpha+n}},&\text{if $\, \alpha>-n$};\\
1+\log \dfrac{1}{[x,y]},&\text{if $\, \alpha=-n$};\\
1,&\text{if $\, \alpha<-n$}.
\end{cases}
\]
\end{lemma}

The next lemma shows that the first part of the above estimate continues to hold when $x$ and $y$ are close enough in the pseudohyperbolic metric. It can be proved along the same lines as \cite[Proposition 5]{M}.

\begin{lemma}\label{Kernel-two-sided}
Let $\alpha>-n$. There exists $0<\delta<1$ such that for every $x\in \mathbb{B}$ and $y\in E_\delta(x)$,
\[
R_\alpha (x,y) \sim \frac{1}{(1-|x|^2)^{\alpha+n}}.
\]
\end{lemma}

\section{ Carleson Measures}\label{s-carleson}
Characterizations of (vanishing) $\alpha$-Carleson measures for weighted harmonic Bergman spaces $b^p_\alpha$ ($\alpha>-1$)  are established by various authors in more general settings. In this subsection we will recall these results.

For $0<\delta<1$ the averaging function $\widehat{\mu}_\delta$ is defined by
\[
\widehat{\mu}_\delta(x)=\frac{\mu(E_\delta(x))}{\nu(E_\delta(x))} \qquad(x\in\mathbb{B}).
\]
More generally, for $\alpha \in \mathbb{R}$ we define
\[
\widehat{\mu}_{\alpha,\delta}(x):=\frac{\mu(E_\delta(x))}{\nu_\alpha(E_\delta(x))} \qquad(x\in\mathbb{B}).
\]
By (\ref{hyper-volume}), $\widehat{\mu}_{\alpha,\delta}(x)\sim \mu(E_\delta(x))/(1-|x|^2)^{\alpha+n}$.

Now, we cite the following characterization of  Carleson measures in terms of averaging
functions which justify the fact that the notion of $\alpha$-Carleson measures on $b^{p}_{\alpha}$ depend only on $\alpha$.

\begin{theorem}\label{Carleson-Besov2}
   Let $\mu$ be a positive Borel measure on $\mathbb{B}$,  $0<p<\infty$ and $\alpha>-1$. The following are equivalent:
  \begin{enumerate}
    \item[(a)] $\mu$ is a  $\alpha$-Carleson measure.
    \item[(b)] $\widehat{\mu}_{\alpha,\delta} \lesssim 1$ for some (every) $0<\delta<1$.

  \end{enumerate}
\end{theorem}

Note that (b) is equivalent to
\[
\mu(E_\delta(x)) \lesssim (1-|x|^2)^{(\alpha+n)} \quad \text{for some (every) $0<\delta<1$.}
\]

\begin{proof}
   Equivalence of (a) and (b) for $\alpha=0$ is proved in \cite[Theorem 3.5]{CLN1} for bounded smooth domains. The proof works equally well for other $\alpha$ too.
\end{proof}
 We also need the following proposition. The proof is similar to \cite[Proposition 3.6]{DOG2}, but for the sake of completeness,we give the details of it.
\begin{proposition}\label{product-carleson}
Let $\mu$ be a positive Borel measure on $\mathbb{B}$. Let $\alpha_{1}>0$ and $-1<\alpha_{2} <\infty$ and
let
\[
\varrho=\alpha_{1}+\alpha_{2}.
\]
If  $\mu$ is a $\varrho$-Carleson measure, then
\begin{equation*}
\int_{\mathbb{B}}|f(x)||g(x)| \, d\mu(x)\lesssim \|f\|_{b^{\infty}_{\alpha_{1}}}\|g\|_{b^{1}_{\alpha_{2}}} \qquad (f\in b^{\infty}_{\alpha_{1}}, g \in b^{1}_{\alpha_{2}}).
\end{equation*}
\end{proposition}
\begin{proof}
Let $f\in b^{\infty}_{\alpha_{1}}, g \in b^{1}_{\alpha_{2}}$. By definitions of $b^{\infty}_{\alpha_{1}}$ and $b^{1}_{\alpha_{2}}$,
\begin{align}
\|fg\|_{L^{1}_{\varrho}}\nonumber
&=\frac{1}{V_{\varrho}}\int_{\mathbb{B}} |f(x)g(x)|(1-|x|^{2})^{\varrho}d\nu(x) \nonumber\\
&\lesssim \|f\|_{b^{\infty}_{\alpha_{1}}}\int_{\mathbb{B}} |g(x)|(1-|x|^{2})^{\alpha_{2}}d\nu(x) \nonumber\\
&\lesssim \|f\|_{b^{\infty}_{\alpha_{1}}}\|g\|_{b^{1}_{\alpha_{2}}}. \label{ineq1}
\end{align}
Thus, $fg\in L^{1}_{\varrho}$. Let $0<\delta<1$. Because $E_{\delta/2}(x)$ is an Euclidean ball with center $c=(1-(\delta/2)^2)x/(1-(\delta/2)^2|x|^2)$ and the radius behaves like $1-|x|^{2}$ when $\delta/2$ is fixed, it follows from \cite[Lemma 3.3]{DOG} that
\begin{equation*}
 |f(x)g(x)|\lesssim  \frac{1}{r^{n}}\int_{B(x,r)} |f(y)g(y)| d\nu(y),
\end{equation*}
whenever $B(x,r)=\{y:|y-x|<r\}\subset E_{\delta/2}(x)$ for all  $x\in \mathbb{B}$. This  immediately leads to the  pointwise
estimate
\begin{equation*}
 |f(x)g(x)|\lesssim  \frac{1}{(1-|x|^2)^{n+\varrho}}\int_{E_{\delta/2}(x)} |f(y)g(y)| (1-|y|^2)^{\varrho}d\nu(y)
\end{equation*}
for all  $x\in \mathbb{B}$. For $a\in \mathbb{B}$ and $x \in E_{\delta/2}(a)$, we note that $E_{\delta/2}(x) \subset E_{\delta}(a)$. Let $E_{\delta/2}(a_{k})$ be the associated sets to the sequence $\{a_k\}=\{a_k(\delta/2)\}$ in  Lemma \ref{ak}. Thus we have
\begin{align*}
 &|f(x)g(x)|\\
 & \lesssim  \frac{1}{(1-|x|^2)^{n+\varrho}}\int_{E_{\delta/2}(x)} |f(y)g(y)| (1-|y|^2)^{\varrho}d\nu(y)\\
 &\lesssim  \frac{1}{(1-|x|^2)^{n+\varrho}}\int_{E_{\delta}(a_{k})} |f(y)g(y)| (1-|y|^2)^{\varrho}d\nu(y), \quad x\in E_{\delta/2}(a_{k})
\end{align*}
 for $k=1,2, \dots$. Then by Lemma \ref{ak} and Lemma \ref{x-y-close}, we have
\begin{align}
&\int_{\mathbb{B}} |f(x)g(x)| \,  d\mu(x)\nonumber\\
&\lesssim \sum_{k=1}^{\infty}\int_{E_{\delta/2}(a_{k})} |f(x)g(x)| \, d\mu(x) \nonumber\\
&\lesssim \sum_{k=1}^{\infty}\int_{E_{\delta}(a_{k})} |f(y)g(y)| (1-|y|^2)^{\varrho}d\nu(y)\nonumber\\
&\times \int_{E_{\delta/2}(a_{k})}\frac{d\mu(x)}{(1-|x|^2)^{n+\varrho}}\nonumber\\
&\lesssim \sum_{k=1}^{\infty} \frac{\mu(E_{\delta/2}(a_{k}))}{(1-|a_{k}|^2)^{n+\varrho}}\int_{E_{\delta}(a_{k})} |f(y)g(y)| (1-|y|^2)^{\varrho}d\nu(y).\label{ineq2}
\end{align}
 Since $\mu$ is a $(1,\varrho)$-Bergman-Carleson measure, by Theorem \ref{Carleson-Besov2} we have
\begin{equation*}
 \mu(E_{\delta/2}(a_{k}))\lesssim (1-|a_{k}|^{2})^{n+\varrho}.
\end{equation*}
Then it follows from this together with (\ref{ineq2}) and Lemma \ref{ak} that
\begin{align}
\int_{\mathbb{B}}|f(x)g(x)| \, d\mu(x)&\lesssim \sum_{k=1}^{\infty} \int_{E_{\delta}(a_{k})} |f(y)g(y)| (1-|y|^2)^{\varrho}d\nu(y)\nonumber \\
&\leq N \|fg\|_{L^{1}_{\varrho}}\lesssim \|fg\|_{L^{1}_{\varrho}} \label{ineq3},
\end{align}
where $N$ is the number provided by  Lemma \ref{ak}. Com\-bin\-ing (\ref{ineq1}) and (\ref{ineq3}) concludes the proof.
\end{proof}

Now we will characterize vanishing $\alpha$-Carleson measures.

\begin{theorem}\label{Carleson-Besov3}
   Let $\mu$ be a positive Borel measure on $\mathbb{B}$,  $0<p<\infty$ and $\alpha>-1$. The following are equivalent:
  \begin{enumerate}
    \item[(a)] $\mu$ is a vanishing $\alpha$-Carleson measure.
    \item[(b)] $\lim_{|x|\to 1^{-}}(1-|x|^2)^{(\alpha+n)(1-\lambda)} \, \widehat{\mu}_{\alpha,\varepsilon}(x) =0 $ for some (every) $0<\varepsilon<1$.
     \item[(c)] $\lim_{k\to \infty}(1-|a_{k}|^{2})^{(n+\alpha)(1-\lambda)}\widehat{\mu}_{\alpha,\delta}(a_{k}) =0$ for some (every) $0<\delta<1$.
  \end{enumerate}
\end{theorem}

\begin{proof}
  Equivalence of (a), (b) and (c)  for $\alpha=0$  is proved in \cite[Theorem 3.5]{CLN2} for bounded smooth domains.  The proof works equally well for other $\alpha$ too.
\end{proof}

\section{Proofs of Theorems \ref{Theorem-1} and \ref{Theorem-2} } \label{proof1}
In this section we will prove Theorems \ref{Theorem-1} and \ref{Theorem-2}. Before that we present a very useful intertwining relation for transforming certain
problems for Toeplitz operators between weighted Bloch spaces $b^{\infty}_{\alpha}$, $\alpha\in \mathbb{R}^{n}$ to similar problems for classical Toeplitz operators between weighted Bloch spaces when $\alpha>0$ . The harmonic and holomorphic Bergman-Besov space versions are in \cite{DOG2} and \cite{AK}, respectively.
\begin{theorem}\label{intertwining relation}
We have $D^t_s ({_{s,t}}T_{\mu})= ({_{s+t}}T_{\kappa}) D^t_s $, where
\[
{_{s+t}}T_{\kappa}f(x)=\frac{1}{V_{s}} \int_{\mathbb{B}} R_{s+t}(x,y) f(y) d\kappa(y)
\]
 is classical Toeplitz operator from  $b^{\infty}_{\alpha_{1}+t}$ to $b^{\infty}_{\alpha_{2}+t}$. Consequently,
 \[
({_{s,t}}T_{\mu})=D^{-t}_{s +t}({_{s+t}}T_{\kappa}) D^t_{s}, \qquad ({_{s+t}}T_{\kappa})=D^{t}_{s}({_{s,t}}T_{\mu})D^{-t}_{s+t}.
\]
\end{theorem}
\begin{proof}
By differentiation under the integral sign and (\ref{**}), we have
\begin{align*}
D^t_s ({_{s,t}}T_{\mu}f)(x)&=\frac{1}{V_{s}} \int_{\mathbb{B}} R_{s+t}(x,y) D^t_{s}f(y)  d\kappa(y)\\
&=({_{s+t}}T_{\kappa}) (D^t_sf)(x) \qquad (f\in b^{\infty}_{\alpha_{1}}).
\end{align*}
For the second and third assertions, we note that $(D^t_{s})^{-1}=D^{-t}_{s +t}$ by (\ref{*}).
\end{proof}

By Theorem \ref{intertwining relation}, ${_{s,t}}T_{\mu}$ is bounded from $b^{\infty}_{\alpha_{1}}$ to $b^{\infty}_{\alpha_{2}}$ if and only if  ${_{s+t}}T_{\kappa}$ is bounded from $b^{\infty}_{\alpha_{1}+t}$ to $b^{\infty}_{\alpha_{2}+t}$. With all the preparation done in earlier sections, now we are ready to prove the main result.
\subsection{Proof of Theorem \ref{Theorem-1} }
\textbf{(i) Implies (ii).}
Let ${_{s,t}}T_{\mu}: b^{\infty}_{\alpha_{1}}\to b^{\infty}_{\alpha_{2}}$ is bounded. First note that $[x,y]\gtrsim(1-|x|^2) $ and $[x,y]\gtrsim (1-|y|^2) $ for $x,y \in \mathbb{B}$. Then fix $x\in \mathbb{B}$ and consider $R_{s+t}(x,.)$.  Under the condition $n+s+t>\alpha_{1}+t$ provided by  (\ref{Equ-1}), it is easy to check using Lemma \ref{R-alpha} that $R_{s+t}(x,.) \in b^{\infty}_{\alpha_{1}+t}$ with
\[
\|R_{s+t}(x,.)\|b^{\infty}_{\alpha_{1}+t}\lesssim (1-|x|^{2})^{\alpha_{1}-(n+s)}.
\]
Take $\delta=\delta_{0}$ where $\delta_{0}$ is the number provided by Lemma \ref{Kernel-two-sided}. By Lemma \ref{x-y-close} and Lemma \ref{Kernel-two-sided}, we have
\begin{align*}
\kappa(E_{\delta}(x))&\lesssim \frac{V_{\alpha_{1}}}{V_{s}} (1-|x|^{2})^{2(n+s+t)}\int_{E_{\delta}(x)} |R_{s+t}(x,y)|^{2} d\kappa(y)\\
&\lesssim \frac{V_{\alpha_{1}}}{V_{s}} (1-|x|^{2})^{2(n+s+t)}\int_{\mathbb{B}} |R_{s+t}(x,y)|^{2} d\kappa(y)\\
&=(1-|x|^{2})^{2(n+s+t)}{_{s+t}}T_{\kappa}[R_{s+t}(x,.)](x),
\end{align*}
and therefore
\begin{align*}
\widehat{\kappa}_{\gamma,\delta}(x)&=\frac{\kappa(E_\delta(x))}{\nu_\gamma(E_\delta(x))}\\
&\lesssim (1-|x|^{2})^{2(n+s+t))-(n+\gamma)}{_{s+t}}T_{\kappa}[R_{s+t}(x,.)](x).
\end{align*}

On the other hand, by the definition of $b^{\infty}_{\alpha}, \, \alpha>0$ together with the boundedness of the Toeplitz operator ${_{s+t}}T_{\kappa}$ and an inequality above, we get

\begin{align*}
{_{s+t}}T_{\kappa}[R_{s+t}(x,.)](x)&= |{_{s+t}}T_{\kappa}[R_{s+t}(x,.)](x)|\\
&\lesssim (1-|x|^{2})^{-t-\alpha_{2}}\|{_{s+t}}T_{\kappa}[R_{s+t}(x,.)]\|_{b^{\infty}_{\alpha_{2}+t}}\\
&\lesssim (1-|x|^{2})^{-t-\alpha_{2}}\|{_{s+t}}T_{\kappa}\|\|R_{s+t}(x,.)\|_{b^{\infty}_{\alpha_{1}+t}}\\
&\lesssim (1-|x|^{2})^{-t-\alpha_{2}-n-s+\alpha_{1}}\|{_{s+t}}T_{\kappa}\|,
\end{align*}
where $\|{_{s+t}}T_{\kappa}\|$ denote the operator norm of ${_{s+t}}T_{\kappa}:b^{\infty}_{\alpha_{1}+t}\to b^{\infty}_{\alpha_{2}+t}$.
Combining these estimates we have
\begin{equation*}
\widehat{\kappa}_{\gamma,\delta}(x)\lesssim \|{_{s+t}}T_{\kappa}\|.
\end{equation*}
By Theorem \ref{Carleson-Besov2} this means that $\kappa$ is a $\gamma$-Carleson measure.

\textbf{(ii) Implies (i).}
Now suppose $(ii)$ holds, that is, $\kappa$ is a $\gamma$-Carleson measure. Let  $\alpha'_{2}$ be a number satisfying
\begin{equation}\label{firstss-implies2}
\alpha'_{2}=s-\alpha_{2}>-1.
\end{equation}
 Since $\alpha'_{2}>-1$ and $\alpha_{2}+t>0$, applying Theorem \ref{Theorem-Dual-of-b1q} (with $s'=\alpha'_{2}+\alpha_{2}+t$ and $t'=0$), we get that the dual of $b^{1}_{\alpha'_{2}}$ can be identified with $b^{\infty}_{\alpha_{2}+t}$ under the pairing
\begin{equation*}
[f,g]_{b^2_{s+t}}=\int_{\mathbb{B}}f(x)\overline{g(x)} \, d\nu_{s+t}(x).
\end{equation*}

Let $f\in b^{\infty}_{\alpha_{1}+t}$ and $h\in b^{1}_{\alpha'_{2}}$. By using Fubini’s theorem and
the reproducing formula (1.5) of \cite{GKU2}, since  $\alpha'_{2}>-1$ and $\alpha'_{2}<s+t$ by the $\alpha_{2}+t>0$, we get that
\begin{align*}
[h,{_{s+t}}T_{\kappa}f]_{b^2_{s+t}}& =\frac{1}{V_{s}} \int_{\mathbb{B}}h(y)\int_{\mathbb{B}} R_{s+t}(x,y)\overline{f(x)}d\kappa(x) \, d\nu_{s+t}(y)\\
&=\frac{1}{V_{s}} \int_{\mathbb{B}}\left(\int_{\mathbb{B}} R_{s+t}(x,y) h(y) d\nu_{s+t}(y)\right) \overline{f(x)}\, d\kappa(x)\\
&=\frac{1}{V_{s}} \int_{\mathbb{B}}h(x)\overline{f(x)} \, d\kappa(x).
\end{align*}
The condition for  $\gamma$ in the theorem is equivalent to
\[
\gamma=\alpha_{1}+t+\alpha'_{2}.
\]
Thus, by  Proposition \ref{product-carleson},
\begin{equation*}
|[h,{_{s+t}}T_{\kappa}f]_{b^2_{s+t}}|\lesssim \int_{\mathbb{B}}|h(x)||f(x)| \, d\kappa(x)\lesssim \|f\|_{b^{\infty}_{\alpha_{1}+t}}\|h\|_{b^{1}_{\alpha'_{2}}}.
\end{equation*}
Hence ${_{s+t}}T_{\kappa}$ is bounded from $b^{\infty}_{\alpha_{1}+t}$ to $b^{\infty}_{\alpha_{2}+t}$.

\subsection{Proof of Theorem \ref{Theorem-2} } \label{proof2}

In this subsection we will prove Theorem \ref{Theorem-2}. Note again that by Theorem \ref{intertwining relation}, ${_{s,t}}T_{\mu}$ is compact from $b^{\infty}_{\alpha_{1}}$ to $b^{\infty}_{\alpha_{2}}$ if and only if  ${_{s+t}}T_{\kappa}$ is compact from $b^{\infty}_{\alpha_{1}+t}$ to $b^{\infty}_{\alpha_{2}+t}$.

\textbf{(i) Implies (ii).}
 Since ${_{s+t}}T_{\kappa}$ is compact, then $\|{_{s+t}}T_{\kappa}f_{k}\|_{b^{\infty}_{\alpha_{2}+t}}\to 0$ for any bounded sequence $\{f_{k}\}$ in $b^{\infty}_{\alpha_{1}+t}$ converging to zero uniformly on compact subsets of $\mathbb{B}$. Let $\{a_{k}\}\subset \mathbb{B} $ with $|a_{k}|\to 1^{-}$ and consider the functions
\begin{equation*}
 f_{k}(x)= (1-|a_{k}|^{2})^{n+s-\alpha_{1}}R_{s+t}(x,a_{k}).
\end{equation*}
Due to the conditions on $s$ and Lemma \ref{R-alpha}, since $[x,y]\gtrsim(1-|x|^2) $ and $[x,y]\gtrsim (1-|y|^2) $ for $x,y \in \mathbb{B}$, we have $\sup_{k}\|f_{k}\|_{b^{\infty}_{\alpha_{1}+t}} <\infty$, and it is obvious that $f_{k}$ converges to zero uniformly on compact subsets of $\mathbb{B}$. Hence $\|{_{s+t}}T_{\kappa}f_{k}\|_{b^{\infty}_{\alpha_{2}+t}}\to 0$. Therefore, proceeding as (i) Implies (ii) in Theorem \ref{Theorem-1}, for any $\delta>0$, we get
\begin{align*}
\widehat{\kappa}_{\gamma,\delta}(a_{k})
&\lesssim (1-|a_{k}|^{2})^{2(n+s+t)-(n+\gamma)-(n+s-\alpha_{1})}{_{s+t}}T_{\kappa}f_{k}(a_{k})\\
&= (1-|a_{k}|^{2})^{\alpha_{2}+t}{_{s+t}}T_{\kappa}f_{k}(a_{k})\\
&\lesssim \|{_{s+t}}T_{\kappa}f_{k}\|_{b^{\infty}_{\alpha_{2}+t}}\to 0.
\end{align*}
Thus, by Theorem \ref{Carleson-Besov3}, the measure $\kappa$ be a vanishing $\delta$-Carleson measure.

\textbf{(ii) Implies (i).}
Now suppose $(ii)$ holds, that is, $\kappa$ is a vanishing $\gamma$-Carleson measure. To prove that the operator ${_{s+t}}T_{\kappa}$ is compact, we must show that $\|{_{s+t}}T_{\kappa}f_{k}\|_{b^{\infty}_{\alpha_{2}+t}}\to 0$  for any bounded sequence $\{f_{k}\}$ in $b^{\infty}_{\alpha_{1}+t}$ converging to zero uniformly on compact subsets of $\mathbb{B}$. Similarly, as in the proof of Theorem \ref{Theorem-1}, by duality  we have (the number $\alpha'_{2}$ is the one defined by (\ref{firstss-implies2})
\begin{align*}
\|{_{s+t}}T_{\kappa}f_{k}\|_{b^{\infty}_{\alpha_{2}+p_{2}t}}&\lesssim \sup_{\|h\|_{b^{1}_{\alpha'_{2}}}\leq 1}|[h,{_{s+t}}T_{\kappa}f_{k}]_{b^{2}_{s+t}}|\\
&\leq \sup_{\|h\|_{b^{1}_{\alpha'_{2}}}\leq 1}\int_{\mathbb{B}} |h(x)||f_{k}(x)| d\kappa(x).
\end{align*}
Let $\kappa_{r}=\kappa|_{\mathbb{B}\backslash\mathbb{\overline{B}}_{r}}$, where $0<r<1$ and $\mathbb{B}_{r}=\{x\in \mathbb{B}: |x|<r\}$. Then $\kappa_{r}$ is also a $\gamma$-Carleson measure, and
\begin{equation*}
\lim_{r\to 1} \| \kappa_{r}\|_{\gamma}=0
\end{equation*}
(see, p.$130$ of \cite{CM}).
Let $\vartheta_{r}$ be the measure defined by $$d\vartheta_{r}=\frac{|f_{k}| d\kappa_{r}(x)}{\|f_{k}\|_{b^{\infty}_{\alpha_{1}+t}}}.$$ Using Proposition \ref{product-carleson}, it is easy to check that, $\vartheta_{r}$ is a $\alpha'_{2}$-Carleson measure on $b^{1}_{\alpha'_{2}}$, and
\begin{equation*}
\lim_{r\to 1} \| \vartheta_{r}\|_{\alpha'_{2}}=0.
\end{equation*}
Hence,
\begin{align}\label{B-compactBr}
\int_{\mathbb{B}\backslash\mathbb{\overline{B}}_{r}} |h(x)||f_{k}(x)|d\kappa(x)
&\leq \int_{\mathbb{B}} |h(x)||f_{k}(x)|\kappa_{r}(x)\nonumber\\
&= \|f_{k}\|_{b^{\infty}_{\alpha_{1}+t}} \int_{\mathbb{B}} |h(x)| d\vartheta_{r} \nonumber\\
&\lesssim  \| \vartheta_{r}\|_{\alpha'_{2}}\|f_{k}\|_{b^{\infty}_{\alpha_{1}+t}}\|h\|_{b^{1}_{\alpha'_{2}}} \nonumber\\
&\lesssim  \epsilon
\end{align}
as $r$ sufficiently close to $1$. Fix such an $r$. Since  $\{f_{k}\}$  converges to $0$ uniformly in compact subsets of $\mathbb{B}$, there is a constant $K>0$ such that for any $k>K$, $|f_{k}(x)|<\epsilon$ for any $x\in \mathbb{\overline{B}}_{r}$.
Therefore, using Proposition \ref{product-carleson}, we have
\begin{align}\label{compactBr}
\int_{\mathbb{\overline{B}}_{r}} |h(x)||f_{k}(x)|d\kappa(x)
&\leq \epsilon \int_{\mathbb{B}} |1||h(x)|\kappa(x) \nonumber\\
&\lesssim \epsilon \|1\|_{b^{\infty}_{\alpha_{1}+t}}\|h\|_{b^{1}_{\alpha'_{2}}}\lesssim \epsilon
\end{align}
for any $x\in \mathbb{\overline{B}}_{r}$. Combining (\ref{B-compactBr}) and (\ref{compactBr}), we get that
\begin{equation*}
 \sup_{\|h\|_{b^{1}_{\alpha'_{2}}}\leq 1}\int_{\mathbb{B}} |h(x)||f_{k}(x)| d\kappa(x) \to 0.
\end{equation*}
Thus, $\|{_{s+t}}T_{\kappa}f_{k}\|_{b^{\infty}_{\alpha_{2}+p_{2}t}}\to 0$, finishing the proof.


\end{document}